\documentclass[11pt]{article}

\usepackage{epsfig}
\usepackage{graphicx}
\usepackage{amsbsy}
\usepackage{amsmath}
\usepackage{amsfonts}
\usepackage{amssymb}
\usepackage{textcomp}
\usepackage{hyperref}
\usepackage{aliascnt}

\newcommand{\mcm}[3]{\newcommand{#1}[#2]{{\ensuremath{#3}}}} 

\mcm{\tuple}{1}{\langle #1 \rangle}
\mcm{\name}{1}{\ulcorner #1 \urcorner}
\mcm{\Nbb}{0}{\mathbb{N}}
\mcm{\Zbb}{0}{\mathbb{Z}}
\mcm{\Rbb}{0}{\mathbb{R}}
\mcm{\Cbb}{0}{\mathbb{C}}
\mcm{\Qbb}{0}{\mathbb{Q}}
\mcm{\Acal}{0}{\cal A}
\mcm{\Bcal}{0}{\cal B}
\mcm{\Ccal}{0}{\cal C}
\mcm{\Dcal}{0}{\cal D}
\mcm{\Ecal}{0}{\cal E}
\mcm{\Fcal}{0}{\cal F}
\mcm{\Gcal}{0}{\cal G}
\mcm{\Hcal}{0}{\cal H}
\mcm{\Ical}{0}{\cal I}
\mcm{\Jcal}{0}{\cal J}
\mcm{\Kcal}{0}{\cal K}
\mcm{\Lcal}{0}{\cal L}
\mcm{\Mcal}{0}{\cal M}
\mcm{\Ncal}{0}{\cal N}
\mcm{\Ocal}{0}{{\cal O}}
\mcm{\Pcal}{0}{{\cal P}}
\mcm{\Qcal}{0}{{\cal Q}}
\mcm{\Rcal}{0}{{\cal R}}
\mcm{\Scal}{0}{{\cal S}}
\mcm{\Tcal}{0}{{\cal T}}
\mcm{\Ucal}{0}{{\cal U}}
\mcm{\Vcal}{0}{{\cal V}}
\mcm{\Wcal}{0}{{\cal W}}
\mcm{\Xcal}{0}{{\cal X}}
\mcm{\Ycal}{0}{{\cal Y}}
\mcm{\Mfrak}{0}{\mathfrak M}

\mcm{\restric}{0}{\upharpoonright}
\mcm{\upset}{0}{\uparrow}
\mcm{\onto}{0}{\twoheadrightarrow}
\mcm{\smallNbb}{0}{{\small \mathbb{N}}}
\DeclareMathOperator{\preop}{op}
\mcm{\op}{0}{^{\preop}}

\newcommand{\se}{\subseteq}
{\begin{array}{c}
\setlength{\unitlength}{1em}}%
{\end{array}}

\usepackage{amsthm}

\newcommand{\theoremize}[2]{\newaliascnt{#1}{thm} \newtheorem{#1}[#1]{#2} \aliascntresetthe{#1}}

\theoremstyle{plain}
\newtheorem{thm}{Theorem}[section]
\theoremize{lem}{Lemma}
\theoremize{sublem}{Sublemma}
\theoremize{claim}{Claim}
\theoremize{rem}{Remark}
\theoremize{prop}{Proposition}
\theoremize{cor}{Corollary}
\theoremize{que}{Question}
\theoremize{oque}{Open Question}
\theoremize{con}{Conjecture}

\theoremstyle{definition}
\theoremize{dfn}{Definition}
\theoremize{eg}{Example}
\theoremize{exercise}{Exercise}
\theoremstyle{plain}

\usepackage{verbatim}
\usepackage{enumerate}
\usepackage[all]{xy}

\usepackage{subfig}

\title{\scshape Topological cycle matroids of infinite graphs}

\author{Johannes Carmesin}

\newcommand{\sm}{\setminus}

\newcommand{\ct}{^\complement}

\DeclareMathOperator{\supp}{supp}

\begin{document}
\maketitle
\begin{abstract}
We prove that the topological cycles of an arbitrary infinite graph induce a matroid.
This matroid in general is neither finitary nor cofinitary. 
\end{abstract}

\section{Introduction}
One central aim of infinite matroid theory is to study the connections to infinite graph theory
\cite{ {union2},{BC:determinacy},{BC:PC}, {BCC:graphic_matroids}, {bruhn:wollan_con},{RD:HB:graphmatroids}, {matroid_axioms},{DP:dualtrees}}.
This approach has not only led us to exciting questions about infinite matroids but also has allowed for new perspectives on infinite graph theory.  
This paper is part of that approach: We resolve the question for which graphs the topological cycles induce a matroid.

So far there were many competing notions of topological cycle \cite{RDsBanffSurvey}. 
For each of these notions we determine when the topological cycles induce a matroid.
This investigation leads us to a single notion of topological cycle. 
This notion is strongest in the sense that the theorem that its topological cycles induce a matroid implies the theorems 
about when the other notions induce matroids.
The matroids for this notion are in general neither finitary nor cofinitary and are uncountable in a nontrivial way.

\vspace{0.3 cm}

Let us be more precise:  
Given a graph together with an end boundary, a \emph{topological cycle} is a homeomorphic image of the unit circle
in the topological space consisting of the graph together with the boundary. Depending on which end boundary we consider, we get a different notion of topological cycle.
The topological cycles \emph{induce} a matroid if their edge sets form the set of circuits of a matroid.

For locally finite graphs, all these end boundaries are the same so that in this case there is only one notion of topological cycle.
Bruhn and Diestel showed in this case that the topological cycles induce a matroid by showing that it is the dual of the finite bond matroid \cite{RD:HB:graphmatroids}. 

For arbitrary 2-connected graphs, the dual of the finite bond matroid also allows 
for a description by topological cycles (in the topological space ETOP). 
However, this matroid is isomorphic to the matroid of a countable graph after deleting loops and parallel edges.

Hence in order to construct matroids that are nontrivially uncountable, we have to consider topological cycles of different topological spaces.
One such space is VTOP, which is obtained from the graph by adding the vertex ends.
In \autoref{dom_lad}, we depicted a graph whose topological cycles in VTOP do not induce a matroid.

   \begin{figure} [htpb]   
\begin{center}
   	  \includegraphics[height=3cm]{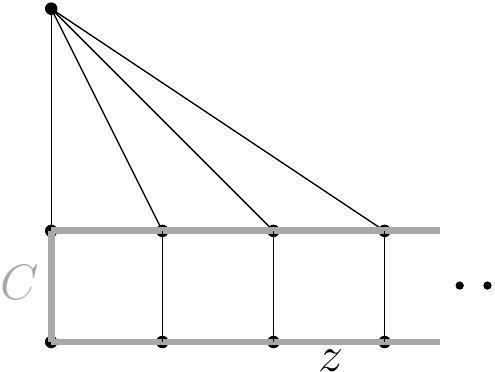}
   	  \caption{The dominated ladder is obtained from the one ended ladder by adding a vertex that is adjacent to every vertex on the upper side of the ladder.
The topological cycles of VTOP of the dominated ladder do not induce a matroid as they violate the elimination axiom (C3):
We cannot eliminate all the triangles from the grey cycle $C$. }
   	  \label{dom_lad}
\end{center}
   \end{figure}

The reason why this example works is that the topological cycle $C$ goes through a dominated end. 
Let DTOP be the topological space obtained from VTOP by deleting the dominated ends. The main result of this paper is the following.

\begin{thm}\label{Psi_version_cor1}
For any graph, the topological cycles in DTOP induce a matroid.
\end{thm}

The matroids of \autoref{Psi_version_cor1} are in general neither finitary nor cofinitary and are uncountable in a nontrivial way. 
The proof of \autoref{Psi_version_cor1} involves a new result on the structure of the end space \cite{C:undom_td_new} and the theory of trees of matroids \cite{BC:TOM2}.

In 1969, Higgs proved that the set of finite cycles and double rays of $G$ is the set of circuits of a matroid if and only if $G$ does not have a subdivision of the Bean-graph \cite{Higgs69_2}.
Using \autoref{Psi_version_cor1}, we get a result for topological cycles of VTOP in the same spirit.

\begin{cor}\label{Psi_version_cor2}\label{Bean-like}
The topological cycles of VTOP induce a matroid if and only if $G$ does not have a subdivision of the dominated ladder, which is depicted in \autoref{dom_lad}. 
\end{cor}

\autoref{Psi_version_cor1} implies similar results concerning the identification space ITOP and we also extend the main result of \cite{BC:determinacy}, see \autoref{apps} for details.

\autoref{Psi_version_cor1} extends to `Psi-Matroids':
Given a set $\Psi$ of ends, by $C_\Psi$ we denote the set of those topological cycles in the topological space obtained from VTOP by deleting the ends not in $\Psi$. 
By $D_\Psi$ we denote the set of those bonds that have no end of $\Psi$ in their closure.
It is not difficult to show that the set of undominated ends is Borel, see \autoref{sec:proof}.
Thus the following is a strengthening of \autoref{Psi_version_cor1}.

\begin{thm}\label{Psi_version}\label{thm:main}
Let $\Psi$ be a Borel set of ends that only contains undominated ends.
Then $C_\Psi$ and $D_\Psi$ are the sets of circuits and cocircuits of a matroid.
\end{thm}

This paper is organised as follows.
After giving the necessary background in \autoref{prelims1}, we prove some intermediate results in \autoref{ends}.
Then we prove \autoref{thm:main} in \autoref{sec:proof}.
Finally, in \autoref{apps} we deduce from it the other theorems mentioned in the Introduction.

\section{Preliminaries}\label{prelims1}
Throughout, notation and terminology for graphs are that of~\cite{DiestelBook10} unless defined differently.
And  $G$ always denotes a graph. We denote the complement of a set $X$ by $X\ct$.
Throughout this paper, \emph{even} always means finite and a multiple of $2$. An edge set $F$ in a graph is a \emph{cut} if there is a partition of the set of vertices such that $F$ is the set of edges with precisely one endvertex in each partition class. 
A vertex set \emph{covers} a cut if every edge of the cut is incident with a vertex of that set.
A cut is \emph{finitely coverable} if there is a finite vertex set covering it.
A \emph{bond} is a minimal nonempty cut. 

For us, a \emph{separation} is just an edge set. 
The \emph{boundary $\partial(X)$} of a separation $X$ is the set of those vertices adjacent with an edge from $X$ and one from $X\ct$.
The \emph{order} of $X$ is the size of $\partial(X)$. 
Given a connected subgraph $C$ of $G$, we denote the set of those edges with at least one endvertex in $C$ by $s_C$.
Given a separation $X$ of finite order and an end $\omega$, then there is a unique component $C$ of $G-\partial(X)$ in which $\omega$ lives.
We say that $\omega$ \emph{lives} in $X$ if $s_C\se X$.

A \emph{tree-decomposition} of $G$ consists of a tree $T$ together with a family of subgraphs $(P_t|t\in V(T))$ of $G$ such that every vertex and edge of $G$ is in at least one of these subgraphs, and such that if $v$ is a vertex of both $P_t$ and $P_w$, then it is a vertex of each $P_u$, where $u$ lies on the $v$-$w$-path in $T$.
Moreover, each edge of $G$ is contained in precisely one $P_t$.
We call the subgraphs $P_t$, the \emph{parts} of the tree-decomposition.
Sometimes, the `Moreover'-part is not part of the definition of tree-decomposition. However, both these two definitions give the same concept of tree-decomposition since any tree-decomposition without this additionally property can easily be changed to one with this property by deleting edges from the parts appropriately.
Given a directed edge $tu$ of $T$, the \emph{separation corresponding to $tu$} is the set of those edges which are in parts $P_w$, where $u$ lies on the unique $t$-$w$-path in $T$.
The \emph{adhesion} of a tree-decomposition is finite if any two adjacent parts intersect finitely.
A key tool in our proof is the main result of \cite{C:undom_td_new}, as follows.

\begin{thm}\label{undom_td}
Every graph $G$ has a tree-decomposition $(T,P_t|t\in V(T))$ of finite adhesion such that
the ends of $T$ are the undominated ends of $G$.
\end{thm}

\begin{rem}\label{rem:connectedness}(\cite[Remark 6.6]{C:undom_td_new})
Let $(T,\leq)$ be the tree order on $T$ as in the proof of \autoref{undom_td} where the root $r$ is the smallest element.
We remark that we constructed $(T,\leq)$ such that $(T,P_t|t\in V(T))$ has the following additional property:
For each edge $tu$ with $t\leq u$, the vertex set $\bigcup_{w\geq u} V(P_w)\sm V(P_t)$ is connected.

Moreover, we construct $(T,P_t|t\in V(T))$ such that if $st$ and $tu$ are edges of $T$ with $s\leq t\leq u$, then $V(P_s)\cap V(P_t)$ and $V(P_t)\cap V(P_u)$ are disjoint.
\end{rem}

Given a part $P_t$ of a tree-decomposition, the \emph{torso} $H_t$ is the multigraph obtained from $P_t$ by adding for each neighbour $u$ of $t$ in the tree a complete graph with vertex set $V(P_t)\cap V(P_u)$.


We denote the set of (vertex-) ends of a graph $G$ by $\Omega(G)$. 
A vertex $v$ is in the \emph{closure of} an edge set $F$ if there is an infinite fan from $v$ to $V(F)$.
An end $\omega$ is in the \emph{closure of} an edge set $F$ if every finite order separation $X$ in which $\omega$ lives meets $F$.
It is straightforward to show that an end $\omega$ is in the closure of an edge set $F$ if and only if every ray (equivalently: some ray) belonging to $\omega$ 
cannot be separated from $F$ by removing finitely many vertices.
An end $\omega$ \emph{lives in} a component $C$ if it is in the closure of the edge set $s_C$.
A \emph{comb} is a subdivision of the graph obtain from the ray by attaching a leaf at each of its vertices. The set of these newly added vertices is the set of  \emph{teeth}. The Star-Comb-Lemma is the following.

\begin{lem}\label{star_comb} (Diestel \cite[Lemma 1.2]{diestel92})
Let $U$ be an infinite set of vertices in a connected graph $G$. 
Then either there is a comb with all its teeth in $U$ 
or a subdivision of the infinite star $S$ with all leaves in $U$. 
\end{lem}

\begin{cor}\label{have_in_closure}
 Every infinite edge set has an end or a vertex in its closure.
\qed\end{cor}

\subsection{Infinite matroids}

An introduction to infinite matroids can be found in \cite{matroid_axioms}, whilst the axiomatisation of infinite matroids we work with here is the one introduced in  \cite{BC:determinacy}.
Let $\Ccal$ and $\Dcal$ be sets of subsets of a groundset $E$, which can be thought of as the sets of circuits and cocircuits of some matroid, respectively.

\begin{itemize}
\item[(C1)] The empty set is not in $\Ccal$.
\item[(C2)] No element of  $\Ccal$ is a subset of another.
        \item[(O1)] $|C\cap D|\neq 1$ for all $C\in \Ccal$ and $D\in \Dcal$. 
        \item[(O2)] For all partitions $E=P\dot\cup Q\dot\cup \{e\}$
either $P+e$ includes an element of $\Ccal$ through $e$ or
$Q+e$ includes an element of $\Dcal$ through $e$.
\end{itemize}

We follow the convention that if we put a $^*$ at an axiom $A$ then this refers to the axiom obtained from $A$ by replacing $\Ccal$ by $\Dcal$, 
for example (C1$^*$) refers to the axiom that the empty set is not in $\Dcal$.
A set $I\se E$ is \emph{independent} if it does not include any nonempty element of $\Ccal$.
Given $X\se E$, a \emph{base of $X$} is a maximal independent subset $Y$ of $X$.
\begin{itemize}
        \item[(IM)] 
Given an independent set $I$ and a superset $X$,
there exists a base of $X$ including $I$.
\end{itemize}

The proof of \cite[Theorem 4.2]{BC:determinacy} also proves the following:

\begin{thm}\label{ortho_axioms+_cor}
Let $E$ be a some set and let $\Ccal,\Dcal\subseteq \Pcal(E)$.
Then there is a matroid $M$ whose set of circuits is $\Ccal$ and whose set of cocircuits is $\Dcal$ if and only if 
 $\Ccal$ and $\Dcal$ satisfy (C1), (C1$^*$), (C2), (C2$^*$), (O1), (O2), and (IM).
\end{thm}

\autoref{ortho_axioms+_cor} shows that the above axioms give an alternative axiomatisation of infinite matroids, which we use in this paper as a definition of infinite matroids. We call elements of $\Ccal$ \emph{circuits} and elements of $\Dcal$ \emph{cocircuits}.
The \emph{dual} of $(\Ccal,\Dcal)$ is the matroid whose set of circuits is $\Dcal$ and whose set of cocircuits is $\Ccal$.

A matroid $(\Ccal,\Dcal)$ is \emph{finitary} if every element of $\Ccal$ is finite, and it is \emph{tame} if every element of $\Ccal$ intersects any element of $\Dcal$ only finitely.
An example of a finitary matroid is the \emph{finite-cycle matroids} of a graph $G$ whose circuits are the edge sets of finite cycles of $G$ and whose cocircuits are the bonds of $G$.
We shall need the following lemma:

\begin{lem}\label{magic_lemma}[\cite[Lemma 2.7]{BCC:graphic_matroids}]
Suppose that $M$ is a matroid, and ${\cal C}$, ${\cal C}^*$ are collections of subsets of $E(M)$ such that ${\cal C}$ contains every circuit of $M$, ${\cal C}^*$ contains every cocircuit of $M$, and for every $o\in {\cal C}$, $b\in {\cal C}^*$, $|o\cap b|\neq 1$. Then the set of minimal nonempty elements of ${\cal C}$ is the set of circuits of $M$ and the set of minimal nonempty elements of ${\cal C}^*$ is the set of cocircuits of $M$.
\end{lem}

\subsection{Trees of presentations}

In this subsection, we give a toy version of the definitions of \cite{BC:TOM2}, which are just enough to state the results of \cite{BC:TOM2} we need in this paper.
A tame matroid is \emph{binary} if every circuit and cocircuit always intersect in an even number of edges.\footnote{In \cite{BC:rep_matroids}, it is shown that most of the equivalent characterisations of finite binary matroids extend to tame binary matroids.}

Roughly, a binary presentation of a tame matroid $M$ is something like a pair of representations over $\mathbb{F}_2$, one of $M$ and of the dual of $M$, formally:

\begin{dfn}
Let $E$ be any set.
A {\em binary presentation} $\Pi$ on $E$ consists of a pair $(V, W)$ of sets of subsets of $E$ satisfying (02) and are orthogonal, that is, 
every $o\in V$ intersects any $d\in W$ evenly. 
We will sometimes denote the first element of $\Pi$ by $V_{\Pi}$ and the second by $W_{\Pi}$. We say that $\Pi$ {\em presents} the matroid $M$ if the circuits of $M$ are the minimal nonempty elements of $V_{\Pi}$ and the cocircuits of $M$ are the minimal nonempty elements of $W_{\Pi}$.
\end{dfn}

Given a finitary binary matroid $M$, let $\overline{V_M}$ be the set of those finite edge sets meeting each cocircuit evenly, and let 
$\overline{W_M}$ be the set of those (finite or infinite) edge sets meeting each circuit evenly.
Then $(\overline{V_M}, \overline{W_M})$ is called the \emph{canonical presentation} of a $M$. 

\begin{dfn}
A {\em tree of binary presentations} $\Tcal$ consists of a tree $T$, together with functions $\overline V$ and $\overline W$ assigning to each node $t$ of $T$ a binary presentation $\Pi(t) = (\overline V(t), \overline W(t))$ on the ground set $E(t)$, such that for any two nodes $t$ and $t'$ of $T$, if $E(t) \cap E(t')$ is nonempty then $tt'$ is an edge of $T$.

For any edge $tt'$ of $T$ we set $E(tt') = E(t) \cap E(t')$. We also define the {\em ground set} of $\Tcal$ to be $E = E(\Tcal) = \left(\bigcup_{t \in V(T)} E(t)\right) \setminus \left(\bigcup_{tt' \in E(T)} E(tt')\right)$. 

We shall refer to the edges which appear in some $E(t)$ but not in $E$ as {\em dummy edges} of $M(t)$: thus the set of such dummy edges is $\bigcup_{tt' \in E(T)} E(tt')$.

A tree of binary presentations is a \emph{tree of binary finitary presentations} if each presentation  $\Pi(t)$ is a canonical presentation of some binary finitary matroid.
\end{dfn}

\begin{dfn}
Let $\Tcal = (T, \overline V, \overline W)$ be a tree of binary presentations. A {\em pre-vector} of $\Tcal$ is a pair $(S, \overline v)$, where $S$ is a subtree of $T$ and $\overline v$ is a function sending each node $t$ of $S$ to some $\overline v(t) \in \overline V(t)$, such that 
for each $t\in S$ we have 
$\overline v(t) \restric_{E(tu)} = \overline v(u) \restric_{E(tu)}\neq 0$ if $u\in S$, and $\overline v(t) \restric_{E(tu)}=0$ otherwise. 

The {\em underlying vector} $\underline{(S, \overline v)}$ of $(S, \overline v)$ 
is the set of those edges in some $\overline v(t)$ for some $t\in V(T)$.
Now let $\Psi$ be a set of ends of $T$. A pre-vector $(S, \overline v)$ is a {\em $\Psi$-pre-vector} if all ends of $S$ are in $\Psi$. 
The space $V_{\Psi}(\Tcal)$ of {\em $\Psi$-vectors} consists of those sets that are a symmetric differences of finitely many underlying vectors of $\Psi$-pre-vectors. 

{\em pre-covectors} are defined like pre-vectors with `$\overline W(t)$' in place of `$\overline V(t)$'.
\emph{underlying covectors} are defined similar to underlying vectors. 
A pre-covector $(S, \overline w)$ is a {\em $\Psi$-pre-covector} if all ends of $S$ are in $\Psi$. 
The space $W_{\Psi}(\Tcal)$ of {\em $\Psi\ct$-covectors} consists of those sets that are a symmetric differences of finitely many underlying covectors of $\Psi\ct$-pre-covectors. 

Finally, $\Pi_{\Psi}(\Tcal)$ is the pair $(V_{\Psi}(\Tcal), W_{\Psi}(\Tcal))$.
\end{dfn}

The following is a consequence of the main result of \cite{BC:TOM2}, Theorem 8.3, and Lemma 6.8.

\begin{thm}[\cite{BC:TOM2}]\label{stellar}
Let $\Tcal=(T,\overline V, \overline W)$ be a tree of binary finitary presentations and $\Psi$ a Borel set of ends of $T$, then $\Pi_\Psi(\Tcal)$ presents a binary matroid.
Moreover, the set of $\Psi$-vectors and $\Psi\ct$-covectors satisfy (O1), (O2) and tameness.
\end{thm}

We shall also need the following related lemma, which is a combination of Lemma 6.6 and Lemma 6.8 from \cite{BC:TOM2}. 

\begin{lem}[\cite{BC:TOM2}]\label{TOM_O1} 
 Let $\Tcal=(T,M)$ be a tree of binary finitary presentations and $\Psi$ be any set of ends of $T$.
Any  $\Psi$-vectors of $\Tcal$ and any $\Psi\ct$-covectors of $\Tcal$ are orthogonal.
\end{lem}

\section{Ends of graphs}\label{ends}

The \emph{simplicial topology of $G$} is obtained from the disjoint union of copies of the unit interval, one for each edge of $G$, by identifying 
two endpoints of these intervals if they correspond to the same vertex. 

First we recall the definition of $|G|$ from \cite{RDsBanffSurvey}, and then we give an equivalent one using inverse limits.
Given a finite set of vertices $S$ and an end $\omega$, by $C(S,\omega)$ we denote the component of $G-S$ in which $\omega$ lives.
Let $\vec{\epsilon}$ be a function from the set of those edges with exactly one endvertex in $C(S,\omega)$ to $(0,1)$.
The set $C_{\vec{\epsilon}}(S,\omega)$ consists of all vertices of $C(S,\omega)$, all ends living in $C(S,\omega)$, the set $e\times (0,1)$ for each edge $e$ with both endvertices in $C(S,\omega)$, together with for each edge $f$ with exactly one endvertex $t(f)$ in $C(S,\omega)$, the set of those points on 
$f\times (0,1)$ with distance less than $\vec{\epsilon}(f)$ from $t(f)$.

The point space of $|G|$ is the union of $\Omega(G)$, the vertex set $V(G)$ and a set $e\times (0,1)$ for each edge $e$ of $G$.
A basis of this topology consists of the sets  $C_{\vec{\epsilon}}(S,\omega)$ together with those sets $O$ that are open considered as sets in the simplicial topology of $G$. Note that $|G|$ is Hausdorff. 

Given a finite vertex set $W$ of $G$, by $G^+[W]$ we denote the (multi-) graph obtained from $G$ by contracting all edges not incident with a vertex of $W$.
Thus the vertex set of $G^+[W]$ is $W$ together with the set of components of $G-W$.
We consider $G^+[W]$ as a topological space endowed with the simplicial topology.
If $U\se W$, then there is a continuous surjective map $f[W,U]$ from $G^+[W]$ to $G^+[U]$.

\begin{thm}\label{inv_limit}
$|G|$ is the inverse limit of the topological spaces $G^+[W]$ with respect to the maps $f[W,U]$. 
\end{thm}

\begin{proof}
For each vertex $v$ of $G$, there is a point in the inverse limit which in the component for $G^+[W]$ takes the vertex whose branch set contains $v$. 
This is the point corresponding to the vertex $v$. Similarly, there are points in the inverse limit corresponding to interior points of edges.
All other points in the inverse limit correspond to havens of order $<\infty$ of $G$. As explained in the appendix of \cite{C:undom_td}, these are precisely the ends of $G$.
Thus $|G|$ and the inverse limit have the same point set. 
It is straightforward to check that they carry the same topology.
\end{proof}

In particular, $|G|$ has the following universal property: Suppose there is a topological space $X$ and for each finite set $W$ of vertices of $G$, a continuous function $f_W:X\to G^+[W]$ such that
$f[W,U] \circ f_W=f_{U}$ for every $U\se W$. Then there is a unique continuous function $f:X\to |G|$ such that 
$\pi_W\circ f=f_W$, where $\pi_W:|G|\to G^+[W]$ is the canonical projection.

A function $f$ from $S^1$ to $|G|$ is \emph{sparse} if $f^{-1}(v)$ never contains more than one point for each interior point $v$ of an edge, and 
if there are two distinct points $x,y\in S^1$ with $f(x)=f(y)$, then there are two points $z_1$ and $z_2$ in different components of $S^1-x-y$ both of whose $f$-values are different from $f(x)$ and not equal to interior points of edges.

Let $f$ from $S^1$ to $|G|$ be a sparse continuous function. Then $f$ meets an edge $e$ in  an interior point if and only if it traverses this edge precisely once.
The set of those edges $e$ is called the \emph{edge set of $f$}, denoted by $E(f)$. If $f$ is a topological cycle, we call $E(f)$ a \emph{topological circuit}. 
An edge set $F$ is \emph{geometrically connected} if $F$ meets every finitely coverable cut $b$ with the property that two components of $G-b$ contain edges of $F$.
Note that if the closure of an edge set $F$ in $|G|$ is connected in $|G|$, then $F$ is geometrically connected. 

\begin{lem}\label{construct_topo_cir_2}
A nonempty edge set $X$ is the set of edges of a sparse continuous function $f$ from  $S^1$ to $|G|$ if and only if it meets every finitely coverable cut evenly and is geometrically connected.
\end{lem}

\begin{proof}
For the `only if'-implication, first note that the edge set of $f$ is geometrically connected since connectedness is preserved under continuous images. 
Second, let $F$ be a finitely coverable cut and let $W$ be a finite vertex set covering it.
If there is a sparse continuous function $f:S^1\to |G|$, then $\pi_W\circ f:G^+[W]\to |G|$ is also continuous and its edge set $Y$ meets $F$ in $X\cap F$.
Note that \autoref{construct_topo_cir_2} is true with `$G^+[W]$' in place of `$|G|$'.
So $X\cap F=Y\cap F$ is even, as $F$ is a cut of $G^+[W]$. 

The `if'-implication is a consequence of \autoref{inv_limit}: Suppose we have a geometrically connected set $X$ meeting every finitely coverable cut evenly.
Then for every finite vertex set $W$, the edge set $X\cap E(G^+[W])$ meets every cut of $G^+[W]$ evenly and is geometrically connected.
Hence $X\cap E(G^+[W])$ is the edge set of a sparse continuous function $f_W$ in $G^+[W]$. 
Each $f_W$ is essentially given by a cyclic order on $E(f_W)$. As each vertex of $W$ is incident with only finitely many vertices of $X$, the set $E(f_W)$ is finite.
Thus we can use a standard compactness argument to ensure that $f_U= f[W,U]\circ f_W$ for every $U\se W$. Then the limit of the $f_W$ is continuous by the universal property of the limit and it is sparse by construction. 
\end{proof}

The simplest example of a finitely coverable cut is the set of edges incident with a fixed vertex. Thus the edge set of a sparse continuous function has even degree at each vertex by \autoref{construct_topo_cir_2}. 
Thus we get the following.
\begin{cor}\label{cir_is_in_F_Psi}
Given a sparse continuous function $f$, then for every finite vertex set $W$ only finitely many components of $G-W$ contain vertices incident with edges of $E(f)$.
\end{cor}

\begin{proof}
Let $X$ be the set of those edges of $E(f)$ incident with vertices of $W$. Note that $X$ is finite by \autoref{construct_topo_cir_2}.
 If two components of $G-W$ contain vertices incident with edges of $E(f)$, 
then $s_D$ intersects $X$ for every component $D$ containing vertices incident with edges of $E(f)$
as $E(f)$ is geometrically connected by \autoref{construct_topo_cir_2}. Thus there are only finitely many such components $D$. 
\end{proof}

Having \autoref{construct_topo_cir_2} and \autoref{cir_is_in_F_Psi} in mind, the set $F$ below can be sought of as the edge set of a topological cycle. Thus the following is an extension of the `Jumping arc'-Lemma \cite{DiestelBook10}:

\begin{lem}\label{cir_meet_bond}
Let $F$ be an edge set meeting every finitely coverable cut evenly such that for every finite vertex set $W$ only finitely many components of $G-W$ contain vertices of $V(F)$.
Let $b$ be a cut which does not intersect $F$ evenly.
Then there is an end in the closure of both $F$ and $b$. 
\end{lem}

Given a finite vertex set $W$ and a component $D$ of $G-W$, we denote by $v(D)$ the vertex of $G^+[W]$ with branch set $D$.

\begin{proof}
First we show that for every finite vertex set $W$ there is a component $D$ of $G-W$ such that $s_{D}$ contains infinitely many edges of both $F$ and $b$.
Suppose for a contradiction there is a vertex set $W$ violating this. For a component $D$ of $G-W$, let $X(D)$ be the set of those vertices in $D$ incident with edges of $b$. 
Similarly, let $Y(D)$ be the set of those vertices in $D$ incident with edges of $F$.
Let $U$ be the union of $W$ with those $X(D)$ such that $Y(D)$ is infinite and those $Y(D)$ such that $Y(D)$ is finite.

By assumption $Y(D)$ is empty for all but finitely many $D$.
Thus $U$ is finite. 
Let $G'$ be the graph obtained from $G^+[U]$ by deleting all vertices $v(K)$ for all components $K$ of $G-U$ such that $Y(K)$ is empty. 

Since $F\cap E(G')$ has even degree at each vertex of $G^+[U]$, the same is true for $G'$.
On the other hand $b\cap E(G')$ is a cut by construction. Thus it intersects  $F\cap E(G')$ evenly. 
As the intersection of $b$ and $F$ is included in $E(G')$ by construction, we get the desired contradiction. 

Hence for every finite vertex set $W$ there is a component $D_W$ of $G-W$ such that $s_{D_W}$ contains infinitely many edges of both $F$ and $b$.
By a standard compactness argument, we can pick the components $D_W$ with the additional property that if $U\se W$, then $f[U,W](v(D_W))=v(D_U)$.
Thus the components $D_W$ define a haven of order $<\infty$ of $G$, which defines an end $\omega$ as explained in the appendix of \cite{C:undom_td}.
By construction the end $\omega$ is in the closure of both $F$ and $b$, completing the proof. 
\end{proof}

\begin{lem}\label{there_is_edge}
 Let $f$ be a sparse continuous function from $S^1$ to $|G|$ and let $x,y\in S^1$ such that $f(x)$ and $f(y)$ are distinct and not interior points of edges. 
Then for each connected component $C$ of $S^1-x-y$ there is an edge $e_C$ of $G$ such that $e_C\times (0,1)$ is included in $f(C)$.
\end{lem}

\begin{proof}
We pick a finite vertex set $W$ containing $x$ and $y$. Clearly, the above lemma is true with `$G^+[W]$' in place of `$|G|$'.
Thus  for each connected component $C$ of $S^1-x-y$ there is an edge $e_C$ of $G$ such that $e_C\times (0,1)$ is included in $\pi_W( f(C))$. 
Hence  $e_C\times (0,1)$ is included in $f(C)$.
\end{proof}

\section{Proof of \autoref{thm:main}}\label{sec:proof}

Given a connected graph $G$, we fix a tree-decomposition $(T,P_t|t\in V(T))$ as in \autoref{undom_td} that has the additional properties of \autoref{rem:connectedness}.
For an undominated end $\omega$ of $G$, we denote the unique end of $T$ in which it lives by $\iota_T(\omega)$.
It is straightforward to check that $\iota_T$ is a homeomorphism from $\Omega(G)$ restricted to the undominated ends to $\Omega(T)$.

For each $t\in V(T)$, let $M(t)$ be the finite-cycle matroid of the torso $H_t$.
Let $\overline{V}(t)=V_{M(t)}$ and  $\overline{W}(t)=W_{M(t)}$. 
Thus $\overline{V}(t)$ consists of those finite edge sets of $H_t$ that have even degree at every vertex, and $\overline{W}(t)$ consists of the cuts of $H_t$.

\begin{rem}
 $\Tcal=(T,\overline{V},\overline{W})$ is a tree of binary finitary presentations.
\qed
\end{rem}

The aim of this section is to prove \autoref{thm:main} from the Introduction.
For that we have to show for each Borel set $\Psi$ of undominated ends of $G$ that
 certain sets $C_\Psi$ and $D_\Psi$ are the sets of circuits and cocircuits of a matroid.
By \autoref{stellar}, we know that $\Pi_{\iota_T(\Psi)}(\Tcal)$ presents some matroid. 
In this section we prove that the circuits and cocircuits of that matroid are given by $C_\Psi$ and $D_\Psi$.

To build this bridge from $\Pi_{\iota_T(\Psi)}(\Tcal)$ to the sets $C_\Psi$ and $D_\Psi$, we start as follows.
We have the two topological spaces $\Omega(G)$ and $\Omega(T)$, which each have their own Borel sets. The next lemma shows that 
these two systems of Borel sets are compatible:

\begin{lem}\label{Borelness_issue}
 The set of dominated ends of $G$ is Borel.
In particular, for any set $\Psi$ of undominated ends, $\Psi$ is Borel in $\Omega(G)$ if and only if $\iota_T(\Psi)$ is Borel in $\Omega(T)$.
\end{lem}

To prove this lemma, we need some intermediate lemmas.
By $B_k(r)$ we denote the ball of radius $k$ around a fixed vertex $r$. 

\begin{lem}\label{rayless}
 The graph $G[B_k(r)]$ has a spanning tree $Y_k$ of diameter at most $2k+1$.
\end{lem}

\begin{proof}
Proving this by induction over $k$, we may assume that  $G[B_{k-1}(r)]$ has a spanning tree $Y_{k-1}$ of diameter at most $2k-1$.
Then $Y_{k-1}$ together with all edges joining vertices in $B_k(r)\sm B_{k-1}(r)$ to vertices in $Y_{k-1}$ is a connected subgraph of 
$G[B_k(r)]$ with vertex set $B_k(r)$. Let $Y_k$ be any of its spanning trees extending $Y_{k-1}$. Moreover, $Y_k$ has diameter at most $2k+1$ by construction.
\end{proof}

\begin{lem}\label{closed_lem}
Let $G$ be a graph with a fixed vertex $r$. 
The set $\Omega_k$ of those ends dominated by some vertex in $B_k(r)$ is closed.
\end{lem}

\begin{proof}
In order to show that $\Omega_k$ is closed, we prove that its complement is open.
For that it suffices to find for each ray $R$ not dominated by some vertex in $B_k(r)$ some finite separator $S_R$ disjoint from $B_k(r)$ that separates $B_k(r)$ from a tail of $R$.

Suppose for a contradiction that there is not such a finite separator $S_R$.
Then we can recursively pick infinitely many $B_k(r)$-$R$-paths that are vertex-disjoint except possibly their starting vertices.
Let $U$ be the set of their starting vertices. The set $U$ is infinite because otherwise some $u\in U$ would dominate $R$, which is impossible.
By \autoref{rayless}, $G[B_k(r)]$ has a rayless spanning tree $Y_k$. 
Applying the Star-Comb-Lemma \cite[Lemma 8.2.2]{DiestelBook10} to $Y_k$ and $U$, we find a vertex $v$ in $G[B_k(r)]$ together with an infinite fan whose endvertices are in $U$.
Enlarging this fan by infinitely many of the previously chosen $B_k(r)$-$R$-paths, yields an infinite fan which witnesses that $v$ dominates $R$, which is the desired contradiction. Thus there is such a finite set $R_S$ for every ray $R$ not dominated by some vertex in $B_k(r)$ and so $\Omega_k$ is closed.
\end{proof}

\begin{proof}[Proof that \autoref{closed_lem} implies \autoref{Borelness_issue}.]
 By \autoref{closed_lem}, the set of dominated ends is a countable union of closed sets and thus Borel.
\end{proof}

The next step in our proof of \autoref{thm:main} is to give a more combinatorial description of the set $C_\Psi$ defined in the Introduction.
For a set $A$, we denote the set of minimal nonempty elements of $A$ by $A^{min}$.
Given a set $\Psi$ of ends of $G$, an edge set $o$ is in $\Ccal_\Psi$ if $o$ meets every finitely coverable cut evenly and is geometrically connected. 
The next lemma implies that  $C_\Psi=\Ccal_\Psi^{min}$.

\begin{lem}\label{C_psi_to_topological_circuits}
Given a Borel set $\Psi$ of ends of $G$, the following are equivalent for some nonempty edge set $o$.
\begin{enumerate}
\item $o\in \Ccal_\Psi$; 
\item $o$ is the edge set of a sparse continuous function from $S^1$ to $|G|$ that only has ends from $\Psi$ in the closure;
\item $o$ is the edge set of a sparse continuous function from $S^1$ to $|G|\sm \Psi\ct$.
\end{enumerate}

In particular, if $o$ is minimal nonempty with one of these properties, then it is minimal nonempty with each of them.
Furthermore $o$ is minimal nonempty with one of these properties if and only if $o$ is the edge set of a topological cycle in $|G|\sm \Psi\ct$.
\end{lem}

\begin{proof}[Proof of \autoref{C_psi_to_topological_circuits}.]
Clearly 2 and 3 are equivalent. And 1 and 2 are equivalent by \autoref{construct_topo_cir_2}. Thus 1,2 and 3 are equivalent. 

To see the `Furthermore'-part, first note that the edge set of a topological cycle in $|G|\sm \Psi\ct$ is a minimal nonempty edge set satisfying 3.
To see the converse, let  $o$ be a minimal edge set which is the edge set of a sparse continuous function $f$ from $S^1$ to $|G|\sm \Psi\ct$.
Suppose for a contradiction that $f$ is not injective. 
Then there are two distinct points $x,y\in S^1$ with $f(x)=f(y)$. By sparseness of $f$, there are two points $z_1$ and $z_2$ in different components of $S^1-x-y$ whose $f$-values are different from $f(x)$.
By \autoref{there_is_edge} applied first to $x$ and $z_1$ and second to $x$ and $z_2$, for each of the two components $C_1$ and $C_2$ of $S^1-x-y$ there is  for each $i=1,2$ an edge $e_i$ of $G$ such that $e_i\times (0,1)$ is included in $f(C_i)$.

We obtain the topological space $K$ from  $C_1\cup\{ x,y\}\se S^1$ by identifying $x$ and $y$. Note that $K$ is homeomorphic to $S^1$.
Moreover, the restriction $\bar f$ of $f$ to $C_1\cup\{ x\}$ considered as a map from $K$ to $|G|$ is continuous.
However, the edge set of $\bar f$ is included in the edge set of $f$ without $e_2$, violating the minimality of the edge set of $f$.
Thus $f$ is injective, and so $o$ is  the edge set of a topological cycle in $|G|\sm \Psi\ct$, completing the proof.
 
\end{proof}

Let $\Dcal_\Psi$ be the set of cuts that do not have an end of $\Psi$ in their closure.
Put another way, $d\in \Dcal_\Psi$ if and only if $d$ does not have an end of $\Psi$ in its closure and 
it meets every finite cycle evenly. Note that $D_\Psi=\Dcal_\Psi^{min}$. 
The next step in our proof of \autoref{thm:main} is to relate $\Ccal_\Psi$ and $\Dcal_\Psi$ to the sets of $\iota_T(\Psi)$-vectors of $\Tcal$ and ${\iota_T(\Psi)}\ct$-covectors of $\Tcal$ .

\begin{lem}\label{precir_lem}\
 {\begin{enumerate}
 \item The edge set of a finite cycle is an underlying vector of an $\emptyset$-pre-vector of $\Tcal$;
 \item Any finitely coverable bond is an underlying covector of an $\emptyset$-pre-covector of $\Tcal$.
 \end{enumerate}}
\end{lem}

\begin{proof}
In this proof we use the tree order $\leq$ on $T$ as in \autoref{rem:connectedness}.

To see the second part, let $d$ be a finitely coverable bond and let $V(G)=A \dot\cup B$ be a partition inducing $d$ and let $A'$ be a finite cover of $d$. Since $G$ is connected, the partition is unique and both $A$ and $B$ are connected. 

For $t\in V(T)$, let $x(t)$ be the set of crossing edges of the partition $V(P_t)=(A\cap V(P_t))\dot\cup (B\cap V(P_t))$ in the torso $H_t$. 
Let $S$ be the set of those nodes such that $A$ and $B$ both meet $V(P_t)$. 

Our aim is to show that $(S,x)$ is an $\emptyset$-pre-covector of $\Tcal$, which then by construction has underlying set $d$.
By construction, $x(t)\in \overline W(t)$. It remains to verify the followings sublemmas.

\begin{sublem}\label{cocir:S_con_and_overlap}
$S$ is connected. Moreover, for each $st\in E(S)$, $x(s)$ contains an edge of the torso $H_t$.  
\end{sublem}

\begin{sublem}\label{cocir:S_rayless}
$S$ is rayless.
\end{sublem} 

\begin{proof}[Proof of \autoref{cocir:S_con_and_overlap}]
It suffices to show for each $st\in E(T)$ separating two vertices of $S$ that $X=V(P_s)\cap V(P_t)$
contains vertices of both $A$ and $B$. This follows from the fact that $A$ and $B$ are both connected and each has vertices in at least two components of
$G-X$.  
\end{proof}

\begin{proof}[Proof of \autoref{cocir:S_rayless}]
Suppose for a contradiction that $S$ includes a ray $v_1v_2\ldots$. By taking a subray if necessary we may assume that $v_i<v_{i+1}$. 
As $A'$ is finite, by the `Moreover'-part of \autoref{rem:connectedness} there is some $m$ such that for all $w\geq v_m$ the part $P_{w}$ does not contain vertices of $A'$. 
By \autoref{rem:connectedness}, $X_i=\left(\bigcup_{w\geq v_{i+1}} V(P_w)\right)\sm V(P_i)$ is connected. As $v_{m+2}\in S$, both $A$ and $B$ contain vertices of $P_{v_{m+2}}\se X_m$. Thus $X_m$ contains an edge of $d$, which is incident with a vertex of $A'$. This is a contradiction to the choice of $m$.  

\end{proof}

To see the first part, let $o$ be the edge set of a finite cycle.
We shall define for each node $t\in V(T)$ an edge set $x(t)$, which plays a similar role as in the last part. For that we need some preparation. 
Let $y(t)=o\cap E(P_t)$. Let $st\in E(T)$ with $s<t$. 
Let $Z(st)$ be the set of those vertices of $V(P_s)\cap V(P_t)$ incident with an odd number of edges of $y(t)$.

\begin{sublem}\label{Zst}
$|Z(st)|$ is even. 
\end{sublem}
\begin{proof}
 The set $b$ of edges joining $V(P_s)\cap V(P_t)$
with $\left(\bigcup_{w\geq t} V(P_w)\right)\sm V(P_s)$ is a cut. 
Thus $o$ intersection $b$ evenly.
Since $b(st)\se E(P_t)$ by \autoref{rem:connectedness}, the number $|Z(st)|$ has the same parity as $|o\cap b|$ and so is even.
\end{proof}

Thus there is a matching $M(st)$ of $Z(st)$ using only edges from $E(H_s)\cap E(H_t)$.
We obtain $x(t)$ from $y(t)$ by adding all the sets $M(st)$ where $s$ is a neighbour of $t$.
Let $S$ be the set of those nodes $t$ where $x(t)$ is nonempty. 

Our aim is to show that $(S,x)$ is an $\emptyset$-pre-vector of $\Tcal$, which then by construction has underlying set $o$.
First note that $S$ is finite as $y(t)$ is nonempty for only finitely many $t$. Thus it remains to verify the following sublemmas.

\begin{sublem}\label{cir:xt_welldef}
$x(t)$ has even degree at each vertex of $H_t$. 
\end{sublem} 

\begin{sublem}\label{cir:S_con_and_overlap}
$S$ is connected. Moreover, for each $st\in E(S)$, $x(s)$ contains an edge of the torso $H_t$.
\end{sublem}

\begin{proof}[Proof of \autoref{cir:xt_welldef}]
By construction $x(t)$ has even degree at all vertices $v$ in $V(H_t)\cap V(H_s)$, where $st\in E(T)$ with $s<t$. 
Hence if $t$ is maximal in $S$, then $x(t)$ has even degree at all vertices of $H_t$.  
Otherwise the statement follows inductively from the statement for all the upper neighbours.
Indeed, let $v\in V(H_t)\sm V(H_s)$, where $st\in E(T)$ with $s<t$. 
Then the degree of $v$ in $x(t)$ is congruent modulo 2 to the degree of $v$ in $o$ plus the sum of the degrees of $v$ in $x(u)$, where the sum ranges over all upper  neighbours $u$ of $t$.
\end{proof}

\begin{proof}[Proof of \autoref{cir:S_con_and_overlap}]
It suffices to show for each $st\in E(T)$ separating two vertices of $S$ that $M(st)$ is nonempty. 
Suppose for a contradiction that $M(st)$ is empty. Let $T_s$ be the component of $T-t$ containing $s$.
The symmetric difference $D_s$ of all $x(u)$ with $u\in T_s$ contains only edges of $o$ and has even degree at each vertex by \autoref{cir:xt_welldef}.

Moreover,  $T_s$ contains a vertex $v$ of $S$. Either $P_v$ contains an edge of $o$ or it has a neighbour $w$ such that $M(vw)$ is nonempty and $P_w$ contains an edge of $o$.
In the later case $w$ is also in $T_s$. So in either case, $D_s$ is nonempty.

Similarly, we define $T_t$ and $D_t$, and deduce that $D_t$ is nonempty. Since $D_s$ and $D_t$ are both nonempty, we deduce that $o$ includes two edge disjoint cycles, which is the desired contradiction.
\end{proof}

\end{proof}

\begin{cor}\label{D_psi}
Every $\Psi\ct$-covector $d$ of $\Tcal$ is in $\Dcal_\Psi$.
\end{cor}

\begin{proof}
First note that $d$ has only ends of $\Psi\ct$ in its closure.
Moreover $d$ is a cut as it meets every finite cycle evenly by  \autoref{precir_lem} and \autoref{TOM_O1} as $\Tcal$ is tree of binary finitary presentations.
\end{proof}

Let $\Fcal_\Psi$ be the set of those edge sets $o$ meeting every finitely coverable cut evenly such that for every finite vertex set $W$
only finitely many components of $G-W$ contain vertices of $V(o)$. Note that $\Ccal_\Psi\se \Fcal_\Psi$ by \autoref{construct_topo_cir_2} and \autoref{cir_is_in_F_Psi}. 

\begin{lem}\label{F_Psi}
Any nonempty $o\in \Fcal_\Psi$ includes a nonempty element of $\Ccal_\Psi$.
Hence, $\Fcal_\Psi^{min}=\Ccal_\Psi^{min}$.
\end{lem}

\begin{proof}
We say that edges $e$ and $f$ of $o$ are in the \emph{same geometric component} if
$o$ meets every finitely coverable cut $d$ such that $e$ and $f$ are in different components of $G-d$.
It is straightforward to check that being in the same geometric component is an equivalence relation. 
Pick some $e\in o$ and let $u$ be its equivalence class. 
It suffices to show that $u$ is in $\Ccal_\Psi$, which is implies by the following two sublemmas.

\begin{sublem}\label{half-finite-cut}
 $u$ is meets every finitely coverable cut evenly.
\end{sublem}

\begin{sublem}\label{geo_con}
 $u$ is geometrically connected.
\end{sublem}

Before proving these two sublemmas, we give a construction that is used in the proof of both these sublemmas.
Let $x\in o$ and let $b$ be a finitely coverable cut. 
For all $z\in b\cap ( o\sm u)$, there is a finitely coverable cut $b_z$ such that $x$ and $z$ are in different components of $G-b_z$.
Let $V(G)=A\dot\cup B$ be a partition inducing $b$, and let $V(G)=A_z\dot\cup B_z$ be a partition inducing $b_z$ such that $x$ has both its endvertices in $A_z$. 
Let $d$ be the cut consisting of those edges with precisely one endvertex in the intersection of $A$ and the finitely many $A_z$. Note that $d$ is finitely coverable.
By construction $d\cap u=d\cap o$. Moreover, $b\cap u=d\cap u$ since any $y\in u$ has both its endvertices in $A_z$.

\begin{proof}[Proof of \autoref{half-finite-cut}]
Let $b$ be a finitely coverable cut. Then $b\cap u=d\cap o$, and thus $b\cap u$ has even size.
\end{proof}

 \begin{proof}[Proof of \autoref{geo_con}]
Let $b$ be a finitely coverable cut such that there are edges $x$ and $y$ of $u$ in different components of $G-b$.
Thus there is a partition $V(G)=A\dot\cup B$ inducing $b$ such that $x$ has both endvertices in $A$ and $y$ has both endvertices in $B$. 
Then $x$ and $y$ are in different components of $G-d$. 
As $x$ and $y$ are in the same geometric component, $d$ meets $o$. Thus $b$ meets $u$, completing the proof.
\end{proof}

\end{proof}

\begin{lem}\label{C_psi}
Every $\Psi$-vector $o$ of $\Tcal$ is in $\Fcal_\Psi$.
\end{lem}

\begin{proof}
The set $o$ meets every finitely coverable bond evenly by  \autoref{precir_lem} and \autoref{TOM_O1} as $\Tcal$ is tree of binary finitary presentations. Since every finitely coverable cut is an edge-disjoint union of finitely many finitely coverable bonds, $o$ meets each finitely coverable cut evenly.

The set $o$ is a finite symmetric difference of sets $o_i$, which are underlying sets of $\Psi$-pre-vectors $(S_i,\overline o_i)$. Note that $S_i$ is locally finite as each $\overline o_i$ is finite and for each $xy\in E(S_i)$, the set $\overline o_i(x)$ contains an edge of the torso of $P_y$. 
It suffices to show that there is no  finite vertex set $W$ together with an infinite set $\Acal$ of components of $G-W$ each containing a vertex of $V(o_i)$.  

Suppose for a contradiction there is such a set $W$.
By the `Moreover'-part of \autoref{rem:connectedness}, there is a rayless subtree $Q$ of $T$ containing all nodes $q$ such that its part $P_q$ contains a vertex of $W$
and the root $r$ of $T$.
For each $A\in \Acal$, there is an edge $z_A$ in $o_i\cap s_A$. Let $t_A$ be the unique node of $T$ such that $z_A\in P_{t_A}$.  

Next we define an edge $e_A$ for each $A\in \Acal$. 
If $t_A\in Q$, we pick $e_A=z_A$. 
Otherwise, let $q_A$ be the last node on the unique $t_A$-$Q$-path and $u_A$ be the node before that. 
By \autoref{rem:connectedness}, $P_{u_A}$ together with the parts above is connected. Thus all these parts are included in $A$. Thus the nodes $u_A$ are distinct for different $A$. Moreover, $q_A$ is on the path from $t_A$ to some $t_B$ for some other $b\in \Acal$. As $S_i$ is connected and $t_A,t_B\in S_i$, 
it must be that $q_A\in S_i$. So $u_A$ is in $S_i$, as well. 
Thus $\bar o_i(q_A)$ contains an edge of the torso of $P_{u_A}$. Pick such an edge for $e_A$.
Summing up, we have picked for each  $A\in \Acal$ an edge $e_A$ in some  $\overline o_i(q)$ with $q\in Q\cap S_i$ such that all these $e_A$ are distinct.

Note that $S_i\cap Q$ is finite as $S_i$ is locally finite and $Q$ is rayless. Since each $e_A$ is in some of the finite sets $\overline o_i(x)$ with $x\in S_i\cap Q$, we get the desired contradiction.
\end{proof}

\begin{thm}\label{technical_variant}
Let $\Psi$ be a Borel set of ends of an infinite connected graph $G$ that are all undominated.
Then there is a matroid $M$ whose set of circuits is  $\Ccal_\Psi^{min}$ and whose set of cocircuits is  $\Dcal_\Psi^{min}$.
\end{thm}

\begin{proof}
By  \autoref{Borelness_issue}, $\iota_T(\Psi)$ is Borel. Thus we apply \autoref{stellar} to the tree of presentations $\Tcal$, yielding that $\Pi_{\iota_T(\Psi)}(\Tcal)$ presents a matroid $M$.
Note that $\Fcal_\Psi$ and $\Dcal_\Psi$ satisfy (01) by \autoref{cir_meet_bond}.
Hence by \autoref{D_psi} and \autoref{C_psi}, we can apply \autoref{magic_lemma} to  $\Fcal_\Psi$ and $\Dcal_\Psi$ and $M$.
As $\Fcal_\Psi^{min}=\Ccal_\Psi^{min}$ by \autoref{F_Psi}, we get the desired result.
\end{proof}

\begin{proof}[Proof of \autoref{thm:main}.]
By considering distinct connected components separately, we may assume that $G$ is connected. 
By \autoref{C_psi_to_topological_circuits}, $\Ccal_\Psi^{min}$ is the set of topological cycles in $|G|\sm \Psi\ct$.
Thus \autoref{thm:main} follows from \autoref{technical_variant}.
\end{proof}


\section{Consequences of \autoref{thm:main}}\label{apps}

First, we prove for any graph $G$ that  the set of topological circuits is the set of circuits of a matroid if and only if $G$ does not have a subdivision of the dominated ladder $H$. This theorem was already mentioned in the Introduction, see \autoref{Bean-like}.
We start with a couple of preliminary lemmas.

\begin{lem}\label{end_degree_one}
Let $\omega$ be a dominated end of a graph $G$ such that there are two vertex-disjoint rays $R$ and $S$
belonging to $\omega$. Then $G$ has a subdivision of $H$.
\end{lem}

\begin{proof}
Let $v$ be a vertex dominating $\omega$. By taking subrays if necessary, we may assume that $v$ lies on neither $R$ nor $S$. As $R$ and $S$ belong to the same end, there are infinitely many vertex-disjoint paths $P_1,P_2,\ldots$ from $R$ to $S$. We may assume that no $P_i$ contains $v$. 
Let $r_i$ be the endvertex of $P_i$ on $R$ and $s_i$ be the endvertex of $P_i$ on $S$.
By taking a subsequence of the $P_i$ if necessary, we can ensure that
the order in which the $r_i$ appear on $R$ is $r_1,r_2,\ldots$.
Similarly, we may assume that 
the order in which the $s_i$ appear on $S$ is $s_1,s_2,\ldots$.

Let $Q_1,Q_2,\ldots$ be an infinite fan from $v$ to $R\cup S$.
So for one of $R$ or $S$, say $R$, there is an infinite fan $Q_1',Q_2',\ldots$ from $v$ to it that avoids the other ray.
As each $P_i$ and each $Q_j'$ is finite, we can inductively construct infinite sets $I,J\se \Nbb$ such that for $i\in I$ and $j\in J$ the paths $P_i$ and $Q_j'$ are vertex-disjoint.

Indeed, just consider the bipartite graph with left hand side $(P_i|i\in \Nbb)$ and right hand side $(Q_j'|j\in \Nbb)$ and put an edge between two paths $P_i$ and $Q_j'$ if they share a vertex. 
Now we use that each vertex of this bipartite graph has only finitely many neighbours
on the other side to construct an independent set of vertices that intersects both sides infinitely. Indeed, for each finite independent set, there are two vertices, one on the left and one on the right, such that the independent set together with these two vertices is still independent. So there is such an infinite independent set and $I$ is its set of vertices on the left and $J$ is its set of vertices on the right.

Finally, $v$ together with $R$, $S$ and $(P_i|i\in I)$ and $(Q_j'|j\in J)$ give rise to a subdivision of $H$, which completes the proof.
\end{proof}

\begin{lem}\label{get_double_ray}
Let $o$ be a topological circuit that has the end $\omega$ in its closure.
Then there is a double ray both of whose tails belong to $\omega$.
\end{lem}

This lemma already was proved in {\cite[Lemma 5.6]{BC:ubiquity}}
in a slightly more general context.

\begin{proof}[Proof of \autoref{Bean-like}.]
 If $G$ has a subdivision of $H$, then as explained in the Introduction the topological set of topological circuits violates (C3).

Thus it remains to consider the case that $G$ has no a subdivision of $H$. Now we apply \autoref{thm:main} with $\Psi$ the set of undominated ends, which is Borel by  \autoref{Borelness_issue}. 

It suffices to show that every topological circuit $o$ of $G$ is a $\Psi$-circuit.
So let $\omega$ be an end in the closure of $o$. Then 
by \autoref{get_double_ray} there is a double ray both of whose tails belong to $\omega$.
If $\omega$ was not in $\Psi$, then $G$ would have a subdivision of $H$ by \autoref{end_degree_one}. Thus $\omega$ is in $\Psi$.
As $\omega$ was arbitrary, this shows that every end in the closure of $o$ is in $\Psi$. 
\end{proof}

\autoref{thm:main} can also be used to extend a central result of
\cite{BC:determinacy}
from countable graphs to graphs with a normal spanning tree as follows.
Given a graph $G$ with a normal spanning tree $T$, in \cite{BC:determinacy} we constructed the 
Undomination graph $U=U(G,T)$. This graph has the pleasant property that it has
few enough edges to have no dominated end but enough edges to have $G$ as a minor. Moreover there is an inclusion $\tilde u$ from the  set of ends of $G$ to the set of ends of $U$.
By \autoref{thm:main}, for every Borel set $\Psi$, the $\Psi$-circuits of $U(G,T)$ are the circuits of a matroid. Now we use the following theorem.

\begin{thm}[{\cite[Theorem 9.9]{BC:determinacy}}]\label{loc_fin_to_fin_sep} 
Assume that $(U,{\tilde u(\Psi)})$ induces a matroid $M$.
Then $(G,\Psi)$ induces the matroid $M/C$.
\end{thm}

We refer the reader to {\cite[Section 3]{BC:determinacy}} for a precise definition of
when the pair $(G,\Psi)$ consisting of a graph $G$ and an end set $\Psi$ induces the matroid $M$.
Very very roughly, this says that the set of certain `topological circuits' which only use ends from $\Psi$ is the set of the circuits of $M$. However the topological space taken there is different from the one we take in this paper, so that the definition of topological circuit there does not match with the definition of topological circuit in this paper.   
For example, in this different notion a ray starting at a vertex $v$ may also be a circuit if the end it converges to is in $\Psi$ and dominated by $v$. However these two notions of topological circuit are the same if no vertex is dominated by an end. 
Thus combining \autoref{loc_fin_to_fin_sep} and \autoref{thm:main}, we get the following.

\begin{cor}\label{other_top_cir}
Let $G$ be a graph with a normal spanning tree and $\Psi\se \Omega(G)$ such that
$\tilde u(\Psi)$ is Borel, then $(G,\Psi)$ induces a matroid.
\end{cor}

For example, if we choose $\Psi$ equal to the set of dominated ends, then we get an interesting instance of this corollary: Like \autoref{thm:main}, this gives a recipe to associate a matroid (which we call $M_I(G)$) to every 
graph $G$ that has a normal spanning tree which in general is neither finitary nor cofinitary. These two matroids need not be the same. 
For example, these two matroids differ for the graph obtained from the two side infinite ladder by adding a vertex so that it dominates precisely one of the two ends.

In fact the circuits of the matroid $M_I(G)$ can be described topologically, namely they are the edge sets of topological cycles in the topological space ITOP, see \cite{RDsBanffSurvey} for a definition of ITOP. 
About ITOP, we shall only need the following fact, which is not difficult to prove: 
Given a graph $G$, we denote by $G_I$, the multigraph obtained from $G$ by identifying any two vertices dominating the same end. 
It is not difficult to show that $G$ and $G_I$ have the same topological cycles. 
Thus in order to study when the topological cycles of $G$ induce a matroid, it is enough to study this question for the graphs $G_I$. 
In what follows, we show that the underlying simple graphs $G_I'$ of $G_I$ always has a normal spanning tree.
This will imply the following:

\begin{cor}
The  topological cycles of ITOP induce a matroid for every graph.
\end{cor}

Let $H'$ be the graph obtained from the dominated ladder $H$ by adding a clone of the infinite degree vertex of $H$.  
Note that  $G_I'$ has no subdivision of $H'$. Thus $G_I'$ has a normal spanning tree due to the following criterion:

\begin{thm}[Halin \cite{Halin:NST}]\label{NST}
If $G$ is connected and does not have a subdivision of the completes graph on countably many vertices, then
$G$ has a normal spanning tree. 
\end{thm}

\bibliographystyle{plain}
\bibliography{literatur}

\end{document}